\documentclass[11pt]{amsart}
\usepackage{graphicx,amscd,amsmath,amsfonts,amssymb,geometry}
\usepackage{amssymb}
\usepackage{amsmath}
\usepackage{amsfonts}
\usepackage{color}
\usepackage{graphicx}

\newtheorem{theorem}{Theorem}[section]

\newtheorem{remark}[theorem]{Remark}

\numberwithin{equation}{section}
\DeclareMathOperator{\ext}{ext\,\!}

\begin{document}
\title[Estimates on the norm of polynomials]{Estimates on the norm of polynomials and applications}
\author[G. Ara\'{u}jo]{Gustavo Ara\'{u}jo\textsuperscript{*}}
\address[G. Ara\'{u}jo]{Departamento de Matem\'{a}tica \\
Universidade Federal da Para\'{\i}ba \\
Jo\~{a}o Pessoa - PB \\
58.051-900. Brazil.}
\email{gdasaraujo@gmail.com}

\author[P. Jim\'{e}nez-Rodr\'{\i}guez]{P. Jim\'{e}nez-Rodr\'{\i}guez}
\address[P. Jim\'{e}nez]{Department of Mathematical Sciences \\
Kent State University \\
Kent, OHIO, 44242. USA.\\
}
\email{pjimene1@kent.edu}
\author[G. A. Mu\~{n}oz]{Gustavo A. Mu\~{n}oz-Fern\'{a}ndez%
\textsuperscript{**}}
\address[G. A. Mu\~{n}oz-Fern\'{a}ndez]{Departamento de An\'{a}lisis Matem%
\'{a}tico \\
Facultad de Ciencias Matem\'{a}ticas \\
Plaza de Ciencias 3 \\
Universidad Complutense de Madrid \\
Madrid, 28040 (Spain).}
\email{gustavo\_fernandez@mat.ucm.es}
\author[J. B. Seoane]{Juan B. Seoane-Sep\'{u}lveda\textsuperscript{**}}
\address[J. B. Seoane-Sep\'{u}lveda]{Departamento de An\'{a}lisis Matem\'{a}%
tico \\
Facultad de Ciencias Matem\'{a}ticas \\
Plaza de Ciencias 3 \\
Universidad Complutense de Madrid \\
Madrid, 28040 (Spain).}
\email{jseoane@mat.ucm.es}
\thanks{\textsuperscript{*}Supported by PDSE/CAPES 8015/14-7.}
\thanks{\textsuperscript{**}Supported by the Spanish Ministry of Science
and Innovation, grant MTM2012-34341.}

\begin{abstract}
In this paper, equivalence constants between various polynomial norms are calculated. As an application, we also obtain sharp values of the Hardy--Littlewood constants for $2$-homogeneous polynomials on $\ell_p^2$ spaces, $2<p\leq\infty$ and lower estimates for polynomials of higher degrees. 
\end{abstract}

\maketitle

\section{Introduction}

Let $\alpha = (\alpha_{1}, \ldots , \alpha _{n}) \in (\mathbb{N}%
\cup\{0\})^{n}$, and define $|\alpha| := \alpha_{1} + \cdots + \alpha_{n}$.
Let $\mathcal{P}(^m\mathbb{R}^n)$ be the finite dimension linear space of
all homogeneous polynomials of degree $m$ on $\mathbb{R}^n$. If $\mathbf{%
x}^{\alpha}$ stands for the monomial $x_{1}^{\alpha _{1}} \cdots
x_{n}^{\alpha_{n}}$ for $\mathbf{x}=(x_{1},\ldots ,x_{n})\in \mathbb{\mathbb{%
R}}^{n}$ and $P\in \mathcal{P}(^m\mathbb{R}^n)$, then $P$ can be written as
\begin{equation}  \label{generalpolynomial}
P(\mathbf{x})=\sum_{|\alpha |=m}a_{\alpha }\mathbf{{x}^{\alpha }.}
\end{equation}
If $|\cdot|$ is a norm on $\mathbb{R}^n$, then
\begin{equation*}
\|P\| := \sup_{x\in B_X}|P(x)|,
\end{equation*}
where $B_X$ is the unit ball of the Banach space $X = (\mathbb{R}^n,|\cdot|)$%
, defines a norm in $\mathcal{P}(^m\mathbb{R}^n)$ usually called polynomial
norm. The space $\mathcal{P}(^m\mathbb{R}^n)$ endowed with the polynomial
norm induced by $X$ is denoted by $\mathcal{P}(^m X)$.

Other norms customarily used in $\mathcal{P}(^m\mathbb{R}^n)$ besides the
polynomial norm are the $\ell_q$ norms of the coefficients, i.e., if $P$ is
as in \eqref{generalpolynomial} and $q \geq 1$, then
\begin{equation*}
|P|_q:=\begin{cases}
\left(\sum_{|\alpha|=m}|a_\alpha|^q\right)^{\frac{1}{q}}&\text{if $1\leq q<+\infty$,}\\
\max\{|a_\alpha|:|\alpha|=n\}&\text{if $q=+\infty$,}
\end{cases}
\end{equation*}
defines another norm in $\mathcal{P}(^m\mathbb{R}^n)$.

The polynomial norm $\|P\|$ is most of the times very difficult to compute, whereas the $\ell_q$ norm of the coefficients $|P|_q$ can be obtained straightforwardly. For this reason it would be convenient to have a good estimate of $\|P\|$ in terms of $|P|_p$. If $\|\cdot\|_p$ represents the polynomial norm of ${\mathcal P}(^m\ell_p^n)$, this paper is devoted to obtain sharp estimates on $\|\cdot\|_p$ ($1\leq p\leq +\infty$) by comparison with the norm $|\cdot|_q$ ($1\leq q\leq +\infty$). Actually since all norms in finite dimensional spaces are equivalent, the polynomial
norm $\|\cdot\|_p$ and the $\ell_q$ norm $|\cdot|_q$ of the coefficients are
equivalent in ${\mathcal P}(^m{\mathbb R}^n)$ for all $1\leq p,q\leq +\infty$, and therefore there exists constants $k>0$ and $K>0$ such that
\begin{equation}  \label{generalproblem}
k\|P\|_p\leq |P|_q \leq K\|P\|_p,
\end{equation}
for all $P\in \mathcal{P}(^m\mathbb{R}^n)$. We will denote the optimal constants in the
above inequalities by $k_{m,n,q,p}$ and $K_{m,n,q,p}$, respectively. Observe that $k_{m,n,q,p}$ is the biggest $k$ fitting in the first inequality in \eqref{generalproblem} whereas $K_{m,n,q,p}$ is the smallest possible $K$ in the second inequality in \eqref{generalproblem}. Since we will study mainly polynomials in ${\mathbb R}^2$, i.e., $n=2$, for the sake of simplicity we will use $k_{m,q,p}$ and $K_{m,q,p}$ instead of $k_{m,2,q,p}$ and $K_{m,2,q,p}$ respectively.
In this paper we calculate the exact values of the constants $k_{m,q,p}$ and $K_{m,q,p}$ for several choices of $m$, $q$ and $p$ with $m\in{\mathbb N}$ and $1\leq q,p\leq +\infty$ using the so called Krein-Milman approach. As a consequence of the Krein-Milman Theorem, for every convex function $f:C\rightarrow {\mathbb R}$ that attains its maximum on a convex set $C$ there exists an extreme point $e$ of $C$ such that $f(e)=\max \{f(x):x\in C\}$. The target function to which the Krein-Milman approach will be applied to calculate $K_{m,q,p}$ is
    $$
    B_{\|\cdot\|_p}\ni P\mapsto |P|_q,
    $$
where $B_{\|\cdot\|_p}$ is the (closed) unit ball of the space ${\mathcal P}(^m\ell_p^2)$.
On the other hand, if $k'_{m,q,p}$ is the smallest $k'$ so that
    $$
    \|P\|_p\leq k'|P|_q,
    $$
for all $P\in {\mathcal P}(^m{\mathbb R}^2)$, it is straightforward that $k_{m,q,p}=\frac{1}{k'_{m,q,p}}$. In in order to calculate $k'_{m,q,p}$ we apply the Krein-Milman approach to the function
    $$
    B_{|\cdot|_q}\ni P\mapsto \|P\|_p,
    $$
where $B_{|\cdot|_q}$ is the unit ball of the space ${\mathcal P}(^m{\mathbb R}^2)\simeq {\mathbb R}^{m+1}$ endowed with the $\ell_q$ norm $|\cdot|_q$.

According to the previous comments, it seems essential to have a complete description of the sets of extreme points of $B_{|\cdot|_q}$ and $B_{\|\cdot\|_p}$, denoted from now on as $\ext(B_{|\cdot|_q})$ and $\ext(B_{\|\cdot\|_p})$ respectively. As for $\ext(B_{|\cdot|_q})$ it is well known that
    $$
    \ext(B_{|\cdot|_q})=\begin{cases}
    \{\pm e_k:\ 1\leq k\leq m+1\}&\text{if $q=1$},\\
    \{\sum_{k=1}^{m+1} \epsilon_k e_k:\ \epsilon_k=\pm 1\}&\text{if $q=+\infty$},\\
    S_{|\cdot|_q}&\text{if $1<q<+\infty$},
    \end{cases}
    $$
where $\{e_1,\ldots,e_{m+1}\}$ is the canonical basis of ${\mathbb R}^{m+1}$ and $S_{|\cdot|_q}$ is the unit sphere of $({\mathbb R}^{m+1},|\cdot|_q)$. On the other hand, the set $\ext(B_{\|\cdot\|_p})$ has also been studied by several authors in \cite{choi1,choi2,G,qm,grecu} and will be explicitly stated for the sake of completeness whenever it is used.

The problems we have just stated in the previous paragraphs are closely related to other questions of interest.
For instance, the famous polynomial Bohnenblust-Hille and Hardy-Littlewood constants are defined from the constants $K_{m,n,q,p}$ considered above. The $m$-th polynomial Bohnenblust-Hille constant is nothing but an upper bound on $K_{m,n,\frac{2m}{m+1},\infty}$, for $n\in{\mathbb N}$.
The reason why the specific choice $q=\frac{2m}{m+1}$ and $p=\infty$ is of interest rests on the fact that the set $\{K_{m,n,q,\infty}:\ n\in{\mathbb N}\}$ is bounded if and only if $q\geq \frac{2m}{m+1}$. Hence, if $q\geq \frac{2m}{m+1}$ there exists a constant $C_{m,q}$ depending only on $m$ and $q$ such that
    \begin{equation}\label{equ:BH}
    |P|_\frac{2m}{m+1}\leq C_{m,q}\|P\|_\infty,
    \end{equation}
for all $P\in{\mathcal P}(^m{\mathbb R}^n)$ and every $n\in{\mathbb N}$. This result was proved by Bohnenblust and Hille in 1931 (see \cite{bh}). Observe that a plausible choice for $C_{m,q}$ would be $C_{m,q}=\sup\{K_{m,n,q,\infty}:\ n\in{\mathbb N}\}$. Actually the best (in the sense of smallest) possible choice for $C_{m,q}$ in \eqref{equ:BH} when $q=\frac{2m}{m+1}$ is called the polynomial Bohnenblust-Hille constant. It is interesting to notice that there exists an apparent difference between the polynomial Bohnenblust-Hille constants for real and complex polynomials. For this reason, the polynomial Bohnenblust-Hille constants are usually denoted by $D_{{\mathbb K},m}$, where ${\mathbb K}$ is either the real or complex field.

Also, if we keep $n\in {\mathbb N}$ fixed, the best (smallest) $C_{m,n}$ in
    $$
    |P|_\frac{2m}{m+1}\leq C_{m,n}\|P\|_\infty,
    $$
for all $P\in{\mathcal P}(^m{\mathbb R}^n)$ is denoted by $D_{{\mathbb K},m}(n)$. Observe that $D_{{\mathbb K},m}(n)=K_{m,n,\frac{2m}{m+1},\infty}$ and hence
    $$
    D_{{\mathbb K},m}=\sup\{D_{{\mathbb K},m}(n):\ n\in{\mathbb N}\}=\sup\left\{K_{m,n,\frac{2m}{m+1},\infty}:\ n\in{\mathbb N}\right\}.
    $$
The calculation of the Bohnenblust-Hille constants $D_{{\mathbb K},m}$ and $D_{{\mathbb K},m}(n)$ has motivated a large amount of papers (see, for example, \cite{annals2011}), but their exact values are still unknown except for very restricted choices of $m$'s and $n$'s. The best lower and upper estimates on $D_{{\mathbb K},m}$ and $D_{{\mathbb K},m}(n)$ known nowadays can be found in \cite{baypelseo,campos,jimmunmurseo}.

A similar result to that of Bohnenblust-Hille can be proved for other values of $p$ different from $\infty$. Indeed, there are constants $C_{m,p}$ and $D_{m,p}$ independent from $n$ such that
\begin{align}
|P|_{\frac{p}{p-m}}&\leq C_{m,p}\|P\|_p\quad \text{for }m<p\leq
2m,\label{ali:HL1}\\
|P|_{\frac{2mp}{mp+p-2m}}&\leq D_{m,p}\| P\|_p\quad \text{for }%
2m\leq p \leq \infty,\label{ali:HL2}
\end{align}
for all $P\in {\mathcal P}(^m{\mathbb R}^n)$ and every $n\in{\mathbb N}$. Here we put $\frac{2mp}{mp+p-2m}=\frac{2m}{m+1}$ when $p=\infty$.
Moreover, the exponents $\frac{p}{p-m}$ and $\frac{2mp}{mp+p-2m}$ in
\eqref{ali:HL1} and \eqref{ali:HL2} respectively are optimal in the sense that for $q < \frac{p%
}{p-m}$ or $q<\frac{2mp}{mp+p-2m}$ any constant $H$ fitting in the inequality
\begin{equation*}
|P|_{q}\leq H\| P\|_p
\end{equation*}
for all $P \in \mathcal{P}(^m\mathbb{R}^n)$ depends necessarily on $%
n $. The proof of the previous highly non trivial results can be found in \cite{albuquerque,dimant}. Let us denote by $C_{{\mathbb K},m,p}$ and $D_{{\mathbb K},m,p}$ the best (smallest) possible constants in \eqref{ali:HL1} and \eqref{ali:HL2} respectively, depending on whether we consider real (${\mathbb K}={\mathbb R}$) of complex (${\mathbb K}={\mathbb R}$) polynomials. These constants are called the polynomial Hardy-Littlewood constants. Notice that the polynomial Bohnenblust-Hille constant $D_{{\mathbb K},m}$ coincides with the Hardy-Littlewood constant $D_{{\mathbb K},m,p}$ when $p=\infty$. Similarly as in the Bohnenblust-Hille setting, we define  $C_{{\mathbb K},m,p}(n)$ and $D_{{\mathbb K},m,p}(n)$ as the best (smallest) value of the constants appearing in \eqref{ali:HL1} and \eqref{ali:HL2} respectively, for $n\in{\mathbb N}$ fixed. Observe that $D_{{\mathbb K},m,p}(n)=K_{m,n,\frac{p}{p-m},p}$. Therefore, we have the following equalities for the optimal constants
of the polynomial Hardy--Littlewood constants:
$$
\begin{cases}
C_{{\mathbb K},m,p} = \sup\limits_n K_{m,n,\frac{p}{p-m},p} & \text{ for }m<p\leq 2m, \\
D_{{\mathbb K},m,p} = \sup\limits_n K_{m,n,\frac{mp}{mp+p-2m},p} & \text{ for }2m\leq p
\leq \infty.
\end{cases}
$$
The calculation of the polynomial Hardy-Littlewood constants $C_{{\mathbb K},m,p}$ and $D_{{\mathbb K},m,p}$ has been the objective of steadily increasing number of publications during the last few years. We refer the interested reader to \cite{aps,ajmnpss} and the references therein for a more detailed understanding on this topic.

\section{Equivalence constants $k_{2,q,p}$ and $K_{2,q,p}$ for $q,p\in\{1,\infty\}$}

In order to apply the Krein-Milman approach we need a full description of $\ext(B_{\|\cdot\|_1})$ and $\ext(B_{\|\cdot\|_\infty})$, which is provided in the following results:

\begin{theorem}[Y.S. Choi, S.G. Kim, and H. Ki, \cite{choi2}]
The extreme polynomials of $B_{\|\cdot\|_1}$ are of the form
\begin{itemize}
\item [(a)] $P(x,y)=\pm x^2\pm 2xy\pm y^2$, or
\item [(b)] $P(x,y)=\pm\frac{\sqrt{4|t|-t^2}}{2}(x^2-y^2)+txy$, where $|t|\in(2,4]$.
\end{itemize}
\end{theorem}

\begin{theorem}[Y.S. Choi, S.G. Kim, \cite{choi1}]
The extreme polynomials of $B_{\|\cdot\|_\infty}$ are of the form
\begin{itemize}
\item [(a)] $P(x,y)=\pm x^2$, or
\item [(b)] $P(x,y)=\pm y^2$, or
\item [(c)] $P(x,y)=\pm \left(tx^2-ty^2\pm 2\sqrt{t(1-t)}xy\right)$, where $t\in\left[\frac{1}{2},1\right]$.
\end{itemize}
\end{theorem}

\begin{theorem}
    $$
    k_{2,q,p}=\begin{cases}
    1&\text{if $q=p=1$,}\\
    1&\text{if $q=1$ and $p=\infty$,}\\
    1&\text{if $q=\infty$ and $p=1$,}\\
    \frac{1}{3}&\text{if $q=p=\infty$}.
    \end{cases}
    $$
Extremal polynomials are given in the following list:
    \begin{align*}
        p_{1,1}(x,y)&=\pm x^2,\ \pm y^2,\\
        p_{1,\infty}(x,y)&=\pm x^2,\ \pm y^2,\\
        p_{\infty,1}(x,y)&=\pm x^2\pm y^2\pm xy,\\
        p_{\infty,\infty}(x,y)&=\pm\left(x^2+y^2\pm xy\right).
    \end{align*}
\end{theorem}

\begin{theorem}
    $$
    K_{2,q,p}=\begin{cases}
    2+2\sqrt{2}&\text{if $q=p=1$,}\\
    1+\sqrt{2}&\text{if $q=1$ and $p=\infty$,}\\
    4&\text{if $q=\infty$ and $p=1$,}\\
    1&\text{if $q=p=\infty$}.
    \end{cases}
    $$
Extremal polynomials are given in the following list:
    \begin{align*}
        P_{1,1}(x,y)&=\pm\frac{\sqrt{2}}{2}(x^2-y^2)+(2+\sqrt{2})xy,\\
        P_{1,\infty}(x,y)&=\pm \left(\frac{2+\sqrt{2}}{4}x^2-\frac{2+\sqrt{2}}{4}y^2\pm \frac{\sqrt{2}}{2}xy\right),\\
        P_{\infty,1}(x,y)&=\pm 4xy,\\
        P_{\infty,\infty}(x,y)&=\pm x^2,\ \pm y^2,\ \pm\left(\frac{1}{2}x^2-\frac{1}{2}y^2\pm xy\right).
    \end{align*}
\end{theorem}

\begin{theorem}
For every $q\in [1,\infty)$, let $f_{q,1}:\left[2,4\right]\rightarrow {\mathbb R}$ and  $f_{q,\infty}:\left[\frac{1}{2},1\right]\rightarrow {\mathbb R}$ be given by
    \begin{align*}
        f_{q,1}(t)&= \left(2^{1-q}(4t-t^2)^\frac{q}{2}+t^q\right)^\frac{1}{q},\\
        f_{q,\infty}(t)&=\left(2t^q+2^qt^\frac{q}{2}(1-t)^\frac{q}{2}\right)^\frac{1}{q}.
    \end{align*}
Then
    \begin{align*}
        K_{2,q,1}&=\max\{f_{q,1}(t):\ t\in[2,4]\},\\
        K_{2,q,\infty}&=\max\left\{f_{q,\infty}(t):\ t\in\left[\frac{1}{2},1\right]\right\}.\\
    \end{align*}
Actually $K_{2,q,1}=4$ and $K_{2,q,\infty}=2^\frac{1}{q}$ for every $q\geq 2$.
\end{theorem}

\begin{remark}
The functions $f_{q,1}$ and $f_{q,\infty}$ are possible to optimize numerically. We present below some values obtained by computer. However, for some specific choices of $q$, we are able to give explicitly the point of attainment of the maximum.
\begin{center}
	\begin{tabular}{r|l|l}
		$q \in [1,2]$ & Maximum of $f_{q,1}(t)$ & Maximum of $f_{q,\infty}(t)$ \\
		\hline
		1.00 & $2 (1+\sqrt{2})$ & $1+\sqrt{2}$ \\
		4/3 & $4.11346$ & $1.83737$ \\
		3/2 & $4.02012$ & $1.67869$ \\
		1.75 & $4.00003$ & $1.51651$\\
		2.00 & $4$ & $\sqrt{2}$\\
	\end{tabular}
\end{center}

The above values, as one might imagine, are very hard (when possible!) to obtain. For instance, for $q=4/3$ the maximum of $f_{q,1}(t)$ is attained at, precisely, $$t = \frac{1}{9} \left(2 \sqrt[3]{181+9 \sqrt{273}}+\sqrt[3]{1448-72 \sqrt{273}}+14\right) \approx 3.79842.$$

Also, for $q=3/2$ the maximum is attained at $t \approx 3.941955$, whose exact value is given by
{ \scriptsize$$t =
\frac{1}{15} \left(\sqrt{6 \left(-10\ 3^{2/3} \sqrt[3]{\frac{2}{9+\sqrt{93}}}+5\ 2^{2/3} \sqrt[3]{3 \left(9+\sqrt{93}\right)}+24\right)}+ \right.$$ $$\left. + \sqrt{6 \left(10\ 3^{2/3} \sqrt[3]{\frac{2}{9+\sqrt{93}}}-5\ 2^{2/3} \sqrt[3]{3 \left(9+\sqrt{93}\right)}+204 \sqrt{\frac{6}{-10\ 3^{2/3} \sqrt[3]{\frac{2}{9+\sqrt{93}}}+5\ 2^{2/3} \sqrt[3]{3 \left(9+\sqrt{93}\right)}+24}}+48\right)}+18\right).$$}

For instance, for $q=4/3$ the maximum of $f_{q,\infty}(t)$ is attained at, precisely, $$t = \frac{1}{36} \left(2 \sqrt[3]{107+9 \sqrt{129}}+\sqrt[3]{856-72 \sqrt{129}}+16\right) \approx 0.86783.$$

Also, for $q=3/2$ the maximum is attained at

\scriptsize{$$t = \frac{1}{20} \sqrt{\frac{10 \sqrt[3]{9+\sqrt{273}}}{3^{2/3}}-\frac{40}{\sqrt[3]{3 \left(9+\sqrt{273}\right)}}+1} \quad + $$
	$$+ \frac{1}{2} \sqrt{-\frac{\sqrt[3]{9+\sqrt{273}}}{10\ 3^{2/3}}+\frac{1}{50}+\frac{2}{5 \sqrt[3]{3 \left(9+\sqrt{273}\right)}}+\frac{49}{50 \sqrt{\frac{10 \sqrt[3]{9+\sqrt{273}}}{3^{2/3}}-\frac{40}{\sqrt[3]{3 \left(9+\sqrt{273}\right)}}+1}}}+\frac{9}{20} \approx 0.878721.$$}

\end{remark}

\section{Equivalence constants $k_{2,q,p}$ and $K_{2,q,p}$ for $p\in(1,\infty)$}

\begin{theorem}[B. Grecu \cite{grecu}]
Let $1<p<2$. A $2$-homogeneous polynomial of unit norm $P$ is a extreme point
of the unit ball of $\mathcal{P}(^{2}\ell _{p}^{2})$ if and only if

\begin{itemize}
\item[(i)] $P\left( x,y\right) =ax^{2}+cy^{2}$ where $ac\geq 0$ and $%
\left\vert a\right\vert ^{\frac{p}{p-2}}+\left\vert c\right\vert ^{\frac{p}{%
p-2}}=1$, or
\item[(ii)] $P\left( x,y\right) =\pm \left( \frac{\alpha ^{p}-\beta ^{p}}{%
\alpha ^{2}+\beta ^{2}}\left( x^{2}-y^{2}\right) +2\alpha \beta \frac{\alpha
^{p-2}+\beta ^{p-2}}{\alpha ^{2}+\beta ^{2}}xy\right) $, with $\alpha ,\beta
\geq 0$ and $\alpha ^{p}+\beta ^{p}=1$.
\item[(iii)] $P\left( x,y\right) =\pm \left( \frac{\alpha ^{p}-\beta ^{p}}{%
\alpha ^{2}-\beta ^{2}}\left( x^{2}-y^{2}\right) +2\alpha \beta \frac{\alpha
^{p-2}+\beta ^{p-2}}{\alpha ^{2}-\beta ^{2}}xy\right) $, with $\alpha ,\beta
\geq 0$ and $\alpha ^{p}+\beta ^{p}=1$.
\end{itemize}
\end{theorem}

From Propositions 2.1 and 2.3 in \cite{grecu} we have the extreme
polynomials of $\mathcal{P}(^{2}\ell _{p}^{2})$ for $p>2$.

\begin{theorem}[B. Grecu \cite{grecu}]\label{extremal_geq_2}
Let $p>2$. A $2$-homogeneous polynomial of unit norm $P$ is a extreme point
of the unit ball of $\mathcal{P}(^{2}\ell _{p}^{2})$ if and only if

\begin{itemize}
\item[(i)] $P\left( x,y\right) =ax^{2}+cy^{2}$ where $ac\geq 0$ and $%
\left\vert a\right\vert ^{\frac{p}{p-2}}+\left\vert c\right\vert ^{\frac{p}{%
p-2}}=1$, or

\item[(ii)] $P\left( x,y\right) =\pm \left( \frac{\alpha ^{p}-\beta ^{p}}{%
\alpha ^{2}+\beta ^{2}}\left( x^{2}-y^{2}\right) +2\alpha \beta \frac{\alpha
^{p-2}+\beta ^{p-2}}{\alpha ^{2}+\beta ^{2}}xy\right) $, with $\alpha ,\beta
\geq 0$ and $\alpha ^{p}+\beta ^{p}=1$.
\end{itemize}
\end{theorem}

\begin{theorem}
For every $q\geq 1$ and $p>2$, let $f_{q,p}:[0,1]\rightarrow{\mathbb R}$ be given by
    $$
    f_{q,p}(t)=\left[ 2\left| \frac{2t^{p}-1}{t^{2}+
    \left(1-t^{p}\right)^{\frac{2}{p}}}\right|^{q}+\left(2t
    \left( 1-t^{p}\right)^{\frac{1}{p}}\frac{t^{p-2}+\left( 1-t^{p}\right)^{\frac{p-2}{p}}}{t^{2}+\left( 1-t^{p}\right)^{\frac{2}{p}}}\right)^{q}\right]^{\frac{1}{q}}.
    $$
Then
\begin{equation*}
K_{2,q,p}=\max\{f_{q,p}(t):t\in[0,1]\}.
\end{equation*}
\end{theorem}

%\begin{proof}
%From the previous lemma we have
%\begin{eqnarray*}
%K_{2,2,q,p}&=&\max \left\{ \max_{\overset{ac\geq 0}{\left\vert a\right\vert
%^{\frac{p}{p-2}}+\left\vert c\right\vert ^{\frac{p}{p-2}}=1}}\left(
%\left\vert a\right\vert ^{q}+\left\vert c\right\vert ^{q}\right) ^{\frac{1}{q%
%}},\right. \\
%&&\qquad \qquad \left. \max_{\overset{\alpha ,\beta \geq 0}{\alpha ^{p}{%
%+\beta }^{p}{=1}}}\left[ 2\left| \frac{\alpha ^{p}-\beta ^{p}}{\alpha
%^{2}+\beta ^{2}}\right|^{q}+\left( 2\alpha \beta \frac{\alpha ^{p-2}+\beta
%^{p-2}}{\alpha ^{2}+\beta ^{2}}\right) ^{q}\right] ^{\frac{1}{q}}\right\} \\
%&=&\max \left\{ \max_{\overset{ac\geq 0}{\left\vert a\right\vert ^{\frac{p}{%
%p-2}}+\left\vert c\right\vert ^{\frac{p}{p-2}}=1}}\left( \left\vert
%a\right\vert ^{q}+\left\vert c\right\vert ^{q}\right) ^{\frac{1}{q}},\max_{{%
%\alpha \in }\left[ 0,1\right] }\left[ 2\left| \frac{2\alpha ^{p}-1}{\alpha
%^{2}+\left( 1-\alpha ^{p}\right) ^{\frac{2}{p}}}\right| ^{q}\right. \right.
%\\
%&&\qquad \qquad \left. \left. +\left( 2\alpha \left( 1-\alpha ^{p}\right) ^{%
%\frac{1}{p}}\frac{\alpha ^{p-2}+\left( 1-\alpha ^{p}\right) ^{\frac{p-2}{p}}%
%}{\alpha ^{2}+\left( 1-\alpha ^{p}\right) ^{\frac{2}{p}}}\right) ^{q}\right]
%^{\frac{1}{q}}\right\} \\
%&=&???.
%\end{eqnarray*}
%\end{proof}

\begin{remark}
We have computationally checked that, for every $q\geq 1$ and $1<p<2$, if we define $f_{q,p}:[0,1]\rightarrow{\mathbb R}$ as in the previous theorem, i.e.,
    $$
    f_{q,p}(t)=\left[ 2\left| \frac{2t^{p}-1}{t^{2}+
    \left(1-t^{p}\right)^{\frac{2}{p}}}\right|^{q}+\left(2t
    \left( 1-t^{p}\right)^{\frac{1}{p}}\frac{t^{p-2}+\left( 1-t^{p}\right)^{\frac{p-2}{p}}}{t^{2}+\left( 1-t^{p}\right)^{\frac{2}{p}}}\right)^{q}\right]^{\frac{1}{q}},
    $$
then
\begin{equation*}
K_{2,q,p}=\max\{f_{q,p}(t):t\in[0,1]\}.
\end{equation*}

\end{remark}

%%%%%%%%%%%%%%%%%%%%%%%%%%%%%%%%%%%%%%%%%%%%%%%%%%%%%%%%%%%%%%%%%%%%%%%%%%%%%%%%%%%%%%%%%%%%%%%%
%%%%%%%%%%%%%%%%%%%%%%%%%%%%%%%%%%%%%%%%%%%%%%%%%%%%%%%%%%%%%%%%%%%%%%%%%%%%%%%%%%%%%%%%%%%%%%%%%
\section{Application to the calculation of the polynomial Hardy--Littlewood constants}

The Krein-Milman approach provides a method to calculate Hardy--Littlewood constants at least for the case of $2$-homogeneus polynomials on $\mathbb{R}^2$. Let us take into consideration the following general lower bound for $m\geq 2$ and $2m\leq p<\infty$ (see \cite{ajmnpss}):
\begin{equation*}  \label{previoushl}
D_{\mathbb{R},m,p}\geq \left( \sqrt[16]{2}\right) ^{m}\geq 2^{\frac{%
m^{2}p+10m-p-6m^{2}-4}{4mp}}.
\end{equation*}

\begin{theorem}
\label{newestimateshl} Let $p>2$. Then,

\begin{itemize}
\item[(i)] For $2<p\leq 4$,
\begin{eqnarray*}
C_{\mathbb{R},2,p} & \geq & C_{\mathbb{R},2,p}(2)  =  K_{2,\frac{p}{p-2},p} \\
& = & \max_{{\alpha \in }\left[ 0,1\right] } \left[ 2\left| \frac{2\alpha
^{p}-1}{\alpha ^{2}+\left( 1-\alpha ^{p}\right) ^{\frac{2}{p}}}\right| ^{%
\frac{p}{p-2}}+\left( 2\alpha \left( 1-\alpha ^{p}\right) ^{\frac{1}{p}}%
\frac{\alpha ^{p-2}+\left( 1-\alpha ^{p}\right) ^{\frac{p-2}{p}}}{\alpha
^{2}+\left( 1-\alpha ^{p}\right) ^{\frac{2}{p}}}\right) ^{\frac{p}{p-2}}%
\right] ^{\frac{p-2}{p}}.
\end{eqnarray*}

\item[(ii)] For $4\leq p\leq \infty$,
\begin{eqnarray*}
D_{\mathbb{R},2,p} & \geq & D_{\mathbb{R},2,p}(2)= K_{2,\frac{4p}{3p-4},p} \\
& = & \max_{{\alpha \in }\left[ 0,1\right] }\left[ 2\left| \frac{2\alpha
^{p}-1}{\alpha ^{2}+\left( 1-\alpha ^{p}\right) ^{\frac{2}{p}}}\right| ^{%
\frac{4p}{3p-4}}+\left( 2\alpha \left( 1-\alpha ^{p}\right) ^{\frac{1}{p}}%
\frac{\alpha ^{p-2}+\left( 1-\alpha ^{p}\right) ^{\frac{p-2}{p}}}{\alpha
^{2}+\left( 1-\alpha ^{p}\right) ^{\frac{2}{p}}}\right) ^{\frac{4p}{3p-4}}%
\right] ^{\frac{3p-4}{4p}}.
\end{eqnarray*}
\end{itemize}
\end{theorem}

\begin{proof}
We have
\begin{eqnarray*}
K_{2,\frac{p}{p-2},p} &=&\max \left\{ \max_{\overset{ac\geq 0}{\left\vert
a\right\vert ^{\frac{p}{p-2}}+\left\vert c\right\vert ^{\frac{p}{p-2}}=1}%
}\left( \left\vert a\right\vert ^{\frac{p}{p-2}}+\left\vert c\right\vert ^{%
\frac{p}{p-2}}\right) ^{\frac{p-2}{p}},\right. \\
&&\qquad \qquad \left. \max_{\overset{\alpha ,\beta \geq 0}{\alpha ^{p}{%
+\beta }^{p}{=1}}}\left[ 2\left| \frac{\alpha ^{p}-\beta ^{p}}{\alpha
^{2}+\beta ^{2}}\right|^{\frac{p}{p-2}}+\left( 2\alpha \beta \frac{\alpha
^{p-2}+\beta ^{p-2}}{\alpha ^{2}+\beta ^{2}}\right) ^{\frac{p}{p-2}}\right]
^{\frac{p-2}{p}}\right\} \\
&=&\max \left\{ 1,\max_{{\alpha \in }\left[ 0,1\right] }\left[ 2\left| \frac{%
2\alpha ^{p}-1}{\alpha ^{2}+\left( 1-\alpha ^{p}\right) ^{\frac{2}{p}}}%
\right| ^{\frac{p}{p-2}}\right. \right. \\
&&\qquad \qquad \left. \left. +\left( 2\alpha \left( 1-\alpha ^{p}\right) ^{%
\frac{1}{p}}\frac{\alpha ^{p-2}+\left( 1-\alpha ^{p}\right) ^{\frac{p-2}{p}}%
}{\alpha ^{2}+\left( 1-\alpha ^{p}\right) ^{\frac{2}{p}}}\right) ^{\frac{p}{%
p-2}}\right] ^{\frac{p-2}{p}}\right\} \\
&=&\max_{{\alpha \in }\left[ 0,1\right] } \left[ 2\left| \frac{2\alpha ^{p}-1%
}{\alpha ^{2}+\left( 1-\alpha ^{p}\right) ^{\frac{2}{p}}}\right| ^{\frac{p}{%
p-2}}+\left( 2\alpha \left( 1-\alpha ^{p}\right) ^{\frac{1}{p}}\frac{\alpha
^{p-2}+\left( 1-\alpha ^{p}\right) ^{\frac{p-2}{p}}}{\alpha ^{2}+\left(
1-\alpha ^{p}\right) ^{\frac{2}{p}}}\right) ^{\frac{p}{p-2}}\right] ^{\frac{%
p-2}{p}}
\end{eqnarray*}%
and, similarly,%
\begin{align*}
& D_{\mathbb{\mathbb{R}},2,p}(2)  \notag \\
& = \max \left\{ \max_{a\in \left[ 0,1\right] }\left[ a ^{\frac{4p}{3p-4}%
}+\left( 1- a^{\frac{p}{p-2}}\right) ^{\frac{4p-8}{3p-4}}\right] ^{\frac{3p-4%
}{4p}},\right.  \notag \\
& \qquad\qquad \left. \max_{{\alpha \in }\left[ 0,1\right] }\left[ 2\left|
\frac{2\alpha ^{p}-1}{\alpha ^{2}+\left( 1-\alpha ^{p}\right) ^{\frac{2}{p}}}%
\right| ^{\frac{4p}{3p-4}}+\left( 2\alpha \left( 1-\alpha ^{p}\right) ^{%
\frac{1}{p}}\frac{\alpha ^{p-2}+\left( 1-\alpha ^{p}\right) ^{\frac{p-2}{p}}%
}{\alpha ^{2}+\left( 1-\alpha ^{p}\right) ^{\frac{2}{p}}}\right) ^{\frac{4p}{%
3p-4}}\right] ^{\frac{3p-4}{4p}}\right\}  \\
&= \max_{{\alpha \in }\left[ 0,1\right] }\left[ 2\left| \frac{2\alpha ^{p}-1%
}{\alpha ^{2}+\left( 1-\alpha ^{p}\right) ^{\frac{2}{p}}}\right| ^{\frac{4p}{%
3p-4}}+\left( 2\alpha \left( 1-\alpha ^{p}\right) ^{\frac{1}{p}}\frac{\alpha
^{p-2}+\left( 1-\alpha ^{p}\right) ^{\frac{p-2}{p}}}{\alpha ^{2}+\left(
1-\alpha ^{p}\right) ^{\frac{2}{p}}}\right) ^{\frac{4p}{3p-4}}\right] ^{%
\frac{3p-4}{4p}}.  \notag
\end{align*}
In order to check why the last equality is true, see Figure \ref{fig:comp1},
where the difference $\Psi(p)-\Phi(p)$ has been sketched,
where $\Phi$ and $\Psi$, are, respectively, the two functions in this last expression.
\begin{figure}[tbp]
\centering
\includegraphics[width=0.8\textwidth]{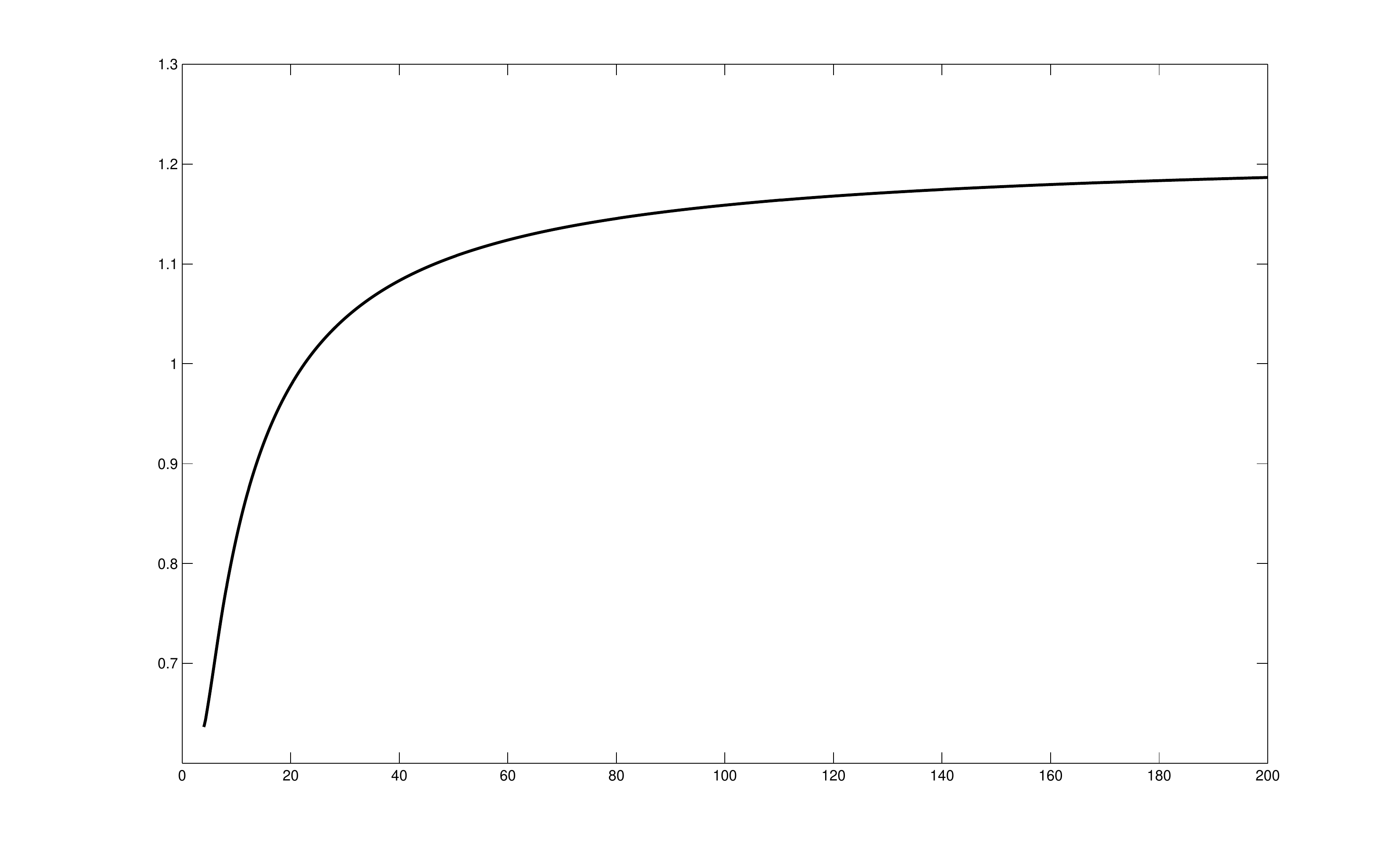}
\caption{Graph of the difference $\Psi(p)-\Phi(p)$.}
\label{fig:comp1}
\end{figure}
\end{proof}

In general, it is not possible to optimize the function appearing in the previous theorem. However, we can obtain the following interesting result for $p=4$.

\begin{theorem}
If $p=4$ we have
$$
D_{\mathbb{R},2,4}(2) = C_{\mathbb{R},2,4}(2) =  K_{2,2,4} = \max_{{\alpha \in }\left[ 0,1\right] }\left[ 2\left\vert \frac{2\alpha^{4}-1}{\alpha ^{2}+\left( 1-\alpha ^{4}\right) ^{\frac{1}{2}}}\right\vert
^{2}+\left( 2\alpha \left( 1-\alpha ^{4}\right) ^{\frac{1}{4}}\right) ^{2}\right] ^{\frac{1}{2}} =\sqrt{2}.
$$
Moreover, surprisingly all the extreme polynomials given in theorem \ref{extremal_geq_2} are also extremal. More explicitly, for the polynomials
$$
P\left( x,y\right) =\pm \left( \frac{\alpha ^{p}-\beta ^{p}}{%
\alpha ^{2}+\beta ^{2}}\left( x^{2}-y^{2}\right) +2\alpha \beta \frac{\alpha
^{p-2}+\beta ^{p-2}}{\alpha ^{2}+\beta ^{2}}xy\right), \text{ with }\alpha ,\beta
\geq 0\text{ and }\alpha ^{p}+\beta ^{p}=1
$$
we obtain
$$
|P|_{2}=\sqrt{2}\|P\|_{4}.
$$
\end{theorem}

Observe that in \cite{cavnunpell} the authors provide numerical lower bound for $D_{\mathbb{R},2,4}(2)$, which is equal to $1.414213562373095$. However, here we obtain that this lower bound is, precisely, $\sqrt{2}$.

\begin{remark} \label{tab_p_2_4} For the case where $2<p <4$ we provide below some numerical values for some choices of $p$ in order to calculate $C_{\mathbb{R},2,p}(2)$. We also give the graphs of the corresponding functions in Figure \ref{fig_p_2_4}.

\begin{center}
	\begin{tabular}{r|c}
		\hline
		$p$ & Maximum \\
		\hline
		2.2 & $1.87786$ \\
		2.4 & $1.78179$ \\
		2.6 & $1.70436$ \\
		2.8 & $1.64067$ \\
		3 & $1.58740$ \\
		3.2 & $1.54221$ \\
		3.4 & $1.50340$ \\
		3.6 & $1.46973$ \\
		3.8 & $1.44024$ \\		
	\end{tabular}
\end{center}

\begin{figure}
	\centering
		\includegraphics[width=0.7\textwidth]{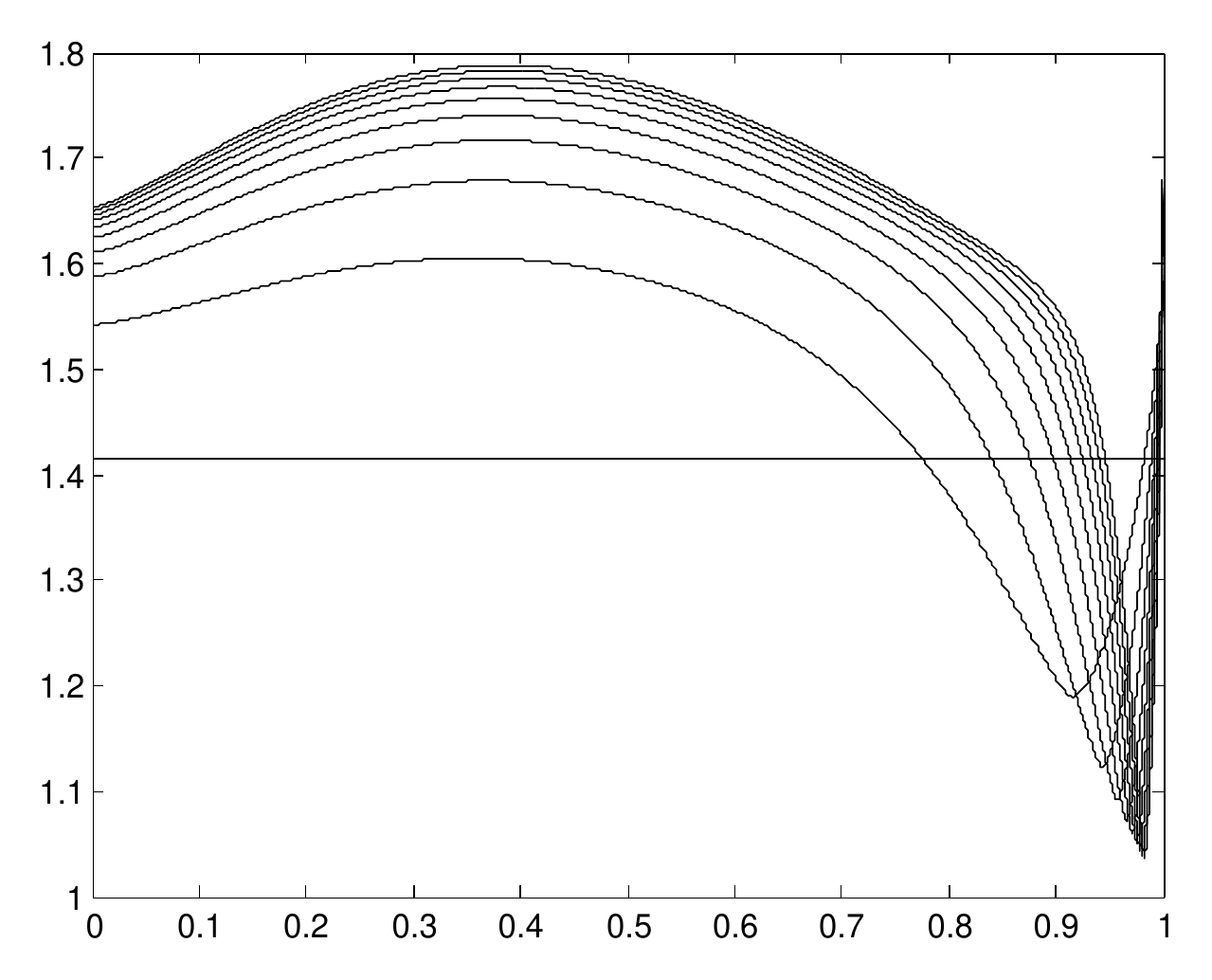}
		\caption{Representation of the functions $f_{\frac{p}{p-2},p}$ for $p$ as in Remark \ref{tab_p_2_4}}
	\label{fig_p_2_4}
\end{figure}

\end{remark}

\begin{remark} The previous function cannot be optimized explicitly in general. We provide a table with some numerical calculations. Notice that, as $p$ increases, the value of the constant $K_{2,{\frac{4p}{3p-4}},p}$ approaches $K_{2,\frac{4}{3},\infty}=D_{\mathbb{R},2}(2) \approx 1.83737$.

\begin{center}
	\begin{tabular}{r|c|c}
		\hline
		$p$ & Maximum & Point of attainment\\
		\hline
		4 & $\sqrt{2}$ & --- \\
		5 & $1.48488$ & $0.99930$ \\
		6 & $1.53632$ & $0.99974$ \\
		7 & $1.57512$ & $0.34940$\\
		8 & $1.60526$ & $0.35688$ \\
		9 & $1.62927$ & $0.36223$ \\
		12 & $1.67869$ & $0.37151$ \\
		25 & $1.75927$ & $0.38261$ \\
		50 & $1.79786$ & $0.38667$ \\
		150 & $1.82410$ & $0.38911$ \\
		250 & $1.82939$ & $0.38957$
	\end{tabular}
\end{center}

\begin{figure}
	\centering
	\begin{tabular}{cc}
		\includegraphics[width=0.5\textwidth]{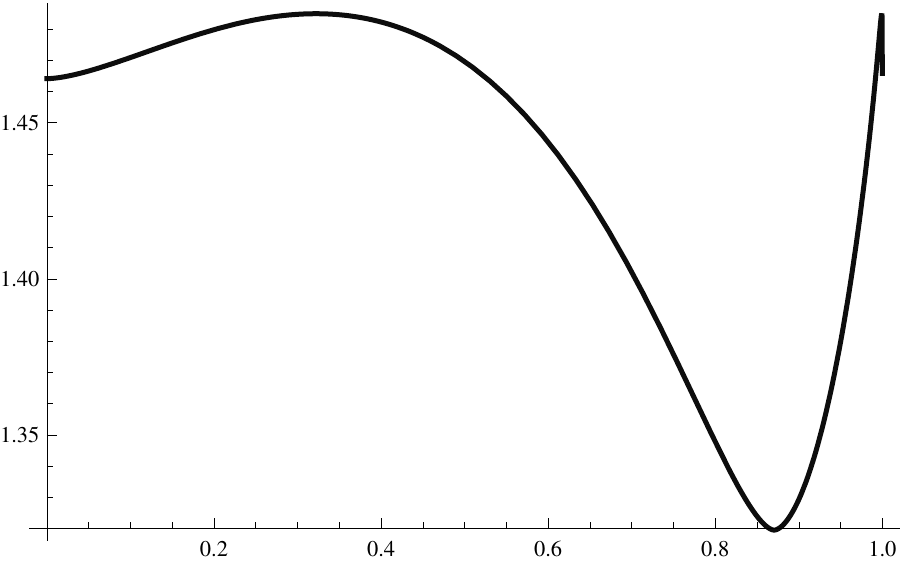} & \includegraphics[width=0.5\textwidth]{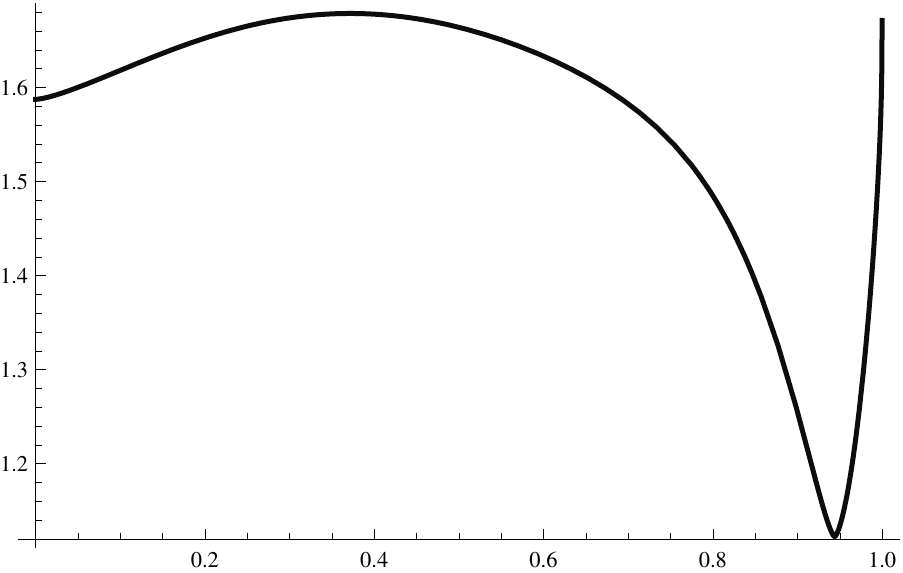}\\
		(a) $p=5$. & (b) $p=12$.
	\end{tabular}
	\label{HL5and12}
\end{figure}
\end{remark}

\section{Lower bounds on the Hardy-Littlewood constants for higher degrees}

In this section we provide a lower bound on $D_{{\mathbb R},2m,4m}(2)=C_{{\mathbb R},2m,4m}(2)$ by considering powers of the extreme polynomials that appear in Theorem \ref{extremal_geq_2}. Observe that if $P$ is as in Theorem \ref{extremal_geq_2} (ii) with $p=4m$, then 
    $$
    \|P^m\|_{4m}=\|P\|^m_{4m}=1,
    $$
and hence 
    $$
    D_{{\mathbb R},2m,4m}(2)=C_{{\mathbb R},2m,4m}(2)\geq \frac{|P^m|_2}{\|P^m\|_{4m}}=|P^m|_2.
    $$
Using MATLAB in order to compute $|P^m|_2$ for large values of $m$
 with $P$ as in Theorem \ref{extremal_geq_2} (ii), we obtain the following estimates:
\begin{center}
    \begin{tabular}{r|r|c}
        \hline
        Degree & $m$ & $D_{{\mathbb R},2m,4m}(2)=C_{{\mathbb R},2m,4m}(2)\geq$\\
        \hline
        4 & $2$ & $1.2937^4$ \\
        8 & $4$ & $1.3880^8$ \\
        20 & $10$ & $ 1.4687^{20}$ \\
        100 & $50$ & $ 1.5303^{100}$\\
        400 & $200$ & $ 1.5465^{100}$\\
        600 & $300$ & $1.5487^{600}$ \\
        800 & $400$ & $1.5498^{800}$
\end{tabular}
\end{center}  
    
Observe that the above values improve the estimates on $D_{{\mathbb R},2m,4m}(2)=C_{{\mathbb R},2m,4m}(2)$ obtained in \cite{cavnunpell} by considering only powers of polynomials of degree $2$. The authors in \cite{cavnunpell} use powers of polynomials of degree greater than 2 in order to obtain a better (bigger) lower estimate on $D_{{\mathbb R},m,2m}(2)=C_{{\mathbb R},m,2m}(2)$. Their strategy consists of considering powers of  polynomials of degrees ranging from 2 to 10 enjoying the same symmetry as the extremal polynomials for the Bohnenblust-Hille constants $D_{{\mathbb R},m}(2)$ $(m=2,3,\ldots,10)$ appearing in \cite{jimmunmurseo}. However, extremal polynomials for $D_{{\mathbb R},m,2m}(2)=C_{{\mathbb R},m,2m}(2)$ may not have the same symmetries as the extremal polynomials for $D_{{\mathbb R},m}(2)$. As a matter of fact, and just to give a numerical hint on the previous comment, if we define
    $$
    P(x,y)=ax^5+bx^4y+cx^3y^2+dx^2y^3+exy^4+fy^5
    $$
with 
    \begin{align*}
    a&=0.000007233947,\\
    b&=0.607036736710,\\
    c&=-0.000044725373,\\
    d&=-0.982210559287,\\
    e&=0.0000283144953,\\
    f&=0.1875854561207,
    \end{align*}
then we obtain
    $$
    D_{{\mathbb R},5,10}(2)=C_{{\mathbb R},5,10}(2)\geq \frac{|P|_2}{\|P\|_{10}}\geq 6.236014.
    $$
However, in \cite{cavnunpell} the authors obtain
    $$
    D_{{\mathbb R},5,10}(2)=C_{{\mathbb R},5,10}(2)\geq 6.191704
    $$
using polynomials with the symmetry 
    $$
    ax^{5}-bx^{4}y-cx^{3}y^{2}+cx^{2}y^{3}+bxy^{4}-ay^{5}.
    $$
Notice that if 
\begin{align*}
a &  =0.19462,\\
b &  =0.66008,\\
c &  =0.97833
\end{align*}
then the polynomial 
    $$
    P_5(x,y)=ax^{5}-bx^{4}y-cx^{3}y^{2}+cx^{2}y^{3}+bxy^{4}-ay^{5}
    $$
is extremal for $D_{{\mathbb R},5}(2)$ (see \cite[Section 3.3]{jimmunmurseo}).


\begin{thebibliography}{9}
\bibitem{albuquerque} N. Albuquerque, F. Bayart, D. Pellegrino, and J.B. Seoane-Sep\'{u}lveda, \textit{Optimal Hardy-Littlewood type inequalities for polynomials and multilinear operators}, Israel J. Math., in press.

\bibitem{aps} G. Ara\'{u}jo, D. Pellegrino, and D. da Silva e Silva, \textit{On the upper bounds for the constants of the Hardy-Littlewood}, J. Funct. Anal. \textbf{267} (2014),
1878--1888.

\bibitem{ajmnpss} G. Ara\'{u}jo, P. Jim\'{e}nez-Rodriguez, G.A.
Mu\~{n}oz-Fernandez, D. N\'{u}\~{n}ez-Alarc\'{o}n, D. Pellegrino, J.B.
Seoane-Sep\'{u}lveda, and D.M. Serrano-Rodr\'iguez, \textit{On the polynomial Hardy--Littlewood inequality}, Arch. Math. \textbf{104} (2015) 259-270.

\bibitem{baypelseo} F. Bayart, D. Pellegrino, and J.B. Seoane-Sep\'{u}lveda, \textit{The Bohr radius of the n-dimensional polydisk is equivalent to $\sqrt{(\log n)/n}$}, Adv. Math. \textbf{264} (2014), 726--746.

\bibitem{bh} H.F. Bohnenblust and E. Hille, \textit{On the absolute convergence of Dirichlet series}, Ann. of Math. (2), \textbf{32}, no. 3 (1931), 600--622.

\bibitem{campos} J.R. Campos, P. Jim\'{e}nez-Rodr\'{\i}guez, G.A. Mu\~{n}oz-Fern\'{a}ndez, D. Pellegrino, and J.B. Seoane-Sep\'{u}lveda, \textit{On the real polynomial Bohnenblust--Hille inequality}, Linear Algebra Appl. \textbf{465} (2015), 391--400.

\bibitem{cavnunpell} W. Cavalcante, D. Nu\~{n}ez-Alarc\'{o}n, D. Pellegrino, \textit{New lower bounds for the constants in the real polynomial Hardy--Littlewood inequality}, arXiv:1506.00159 [math.FA].

\bibitem{choi1} Y.S. Choi and S.G. Kim, \textit{The unit ball of ${\mathcal P}(^2l_2^2)$}, Arch. Math. (Basel), \textbf{71}, no. 6 (1998) 472--480.

\bibitem{choi2} Y.S. Choi, S.G. Kim, and H. Ki. \textit{Extreme polynomials and multilinear forms on $l_1$}, J. Math. Anal. Appl., \textbf{228} no. 2 (1998) 467--482.

\bibitem{annals2011} A. Defant, L. Frerick J. Ortega-Cerd\`{a}, M. Ouna{\"{\i }}es, and K. Seip. \textit{The Bohnenblust-Hille inequality for homogeneous polynomials is hypercontractive}, Ann. of Math. (2), \textbf{174} no. 1 (2011), 485--497.

\bibitem{dimant} V. Dimant and P. Sevilla-Peris, \textit{Summation of coefficients of polynomials on $\ell_p$ spaces}, arXiv:1309.6063v1 [math.FA].

\bibitem{G} B.C. Grecu, {\it Geometry of three-homogeneous polynomials on real Hilbert spaces}, J. Math. Anal. Appl. {\bf 246} (2000) 1, 217--229.

\bibitem{qm} B.C. Grecu, {\it Extreme 2-homogeneous polynomials on Hilbert spaces},  Quaest. Math. \textbf{25}  (2002),  no. 4, 421--435.

\bibitem{grecu} B.C. Grecu, \textit{Geometry of $2$-homogeneous polynomials on $\ell_p$ spaces, $1<p<\infty$}, J. Math. Anal. Appl. \textbf{273}, no. 2 (2002), 262-\^{a}€"282.

%\bibitem{grecu} B.C. Grecu, \textit{Geometry of homogeneous polynomials on two-dimensional real Hilbert spaces}, Journal of Mathematical Analysis and Applications \textbf{293}, no. 2 (2004) 578-588.

\bibitem{jimmunmurseo} P. Jim\'{e}nez-Rodr\'iguez, G.A. Mu\~{n}oz-Fern\'{a}ndez, M. Murillo-Arcila, J.B. Seoane-Sep\'{u}lveda, \textit{Sharp values for the constants in the polynomial
Bohnenblust-Hille inequality}, arXiv:1502.02173 [math.FA].
\end{thebibliography}
\end{document}